\documentclass{amsart} 

\usepackage{amsfonts, amsmath}

\usepackage[utf8]{inputenc}
\usepackage[T1]{fontenc}
\usepackage[english]{babel}
\usepackage{enumitem}
\usepackage{amsthm}
\usepackage{csquotes}
\usepackage{listings}
\newtheorem{lemma}{Lemma}[section]
\newtheorem{proposition}{Proposition}[section]

\DeclareMathOperator{\ord}{ord}

\begin{document}

\title[String C-groups for $Suz$ and $Ru$]{The string C-group representations of the Suzuki and Rudvalis sporadic groups}

\author{Dimitri Leemans}
\address{Dimitri Leemans, Département de Mathématique, Université libre de Bruxelles, C.P.216, Boulevard du Triomphe, 1050 Brussels, Belgium}
\email{dleemans@ulb.ac.be}
\author{Jessica Mulpas}
\address{Jessica Mulpas, Département de Mathématique, Université libre de Bruxelles, C.P.216, Boulevard du Triomphe, 1050 Brussels, Belgium}
\email{jessica.mulpas@outlook.com}
\maketitle
\begin{abstract}

We present new algorithms to classify all string C-group representations of a given group $G$.
We use these algorithms to classify all string C-group representations of the sporadic groups of Suzuki and Rudvalis.
\end{abstract}

\maketitle

\section{Introduction}

A natural way of gaining knowledge on the structure of a group is to search for geometric and combinatorial objects on which it can act. Among those objects, abstract regular polytopes are of great interest for they are highly symmetric combinatorial structures with many singular algebraic and geometric properties. Moreover, due to the well-known one-to-one correspondence between abstract regular polytopes and string C-groups, finding an abstract regular polytope with a given automorphism group $G$ means obtaining a presentation for $G$ as generated by a set of involutions with many convenient properties. Finally, string C-groups are smooth quotients of Coxeter groups, hence any Coxeter diagram found in this process permits to conclude that the group appears as a quotient of the corresponding (infinite) Coxeter group.

In 2006, Dimitri Leemans and Laurence Vauthier published "An atlas of polytopes for almost simple groups" \cite{AtlasSmallGroups}, in which they gave classfications of all abstract regular polytopes with automorphism group an almost simple group with socle of order up to 900,000 elements. The same year, Michael Hartley published "An atlas of small regular polytopes" \cite{AtlasHartley}, where he classified all regular polytopes with automorphism groups of order at most 2000 (not including orders 1024 and 1536). More recently in \cite{HartleyHulpke}, Hartley also classified, together with Alexander Hulpke, all polytopes for the sporadic groups as large as the Held group (of order 4,03,387,200)  and Leemans and Mark Mixer classified, among others, all polytopes for the third Conway group (of order 495,766,656,000) \cite{LastAlgo}. Leemans, Mixer and Thomas Connor then computed all regular polytopes of rank at least four for the $O'Nan$ group \cite{ConnorONan}.

These computational data lead to many theoretical results.
For more details we refer to the recent survey~\cite{LeemansSurvey}.

The recorded computational times and memory usages for the use of previous algorithms on groups among which the $O'Nan$ group \cite{ConnorONan} and the third Conway group 
\cite{LastAlgo} provide a motivation to improve the algorithms in order to obtain computational data for groups currently too large. 

The challenges involve dealing with groups that have a relatively large smallest permutation representation degree as well as groups with large Sylow 2-subgroups. The largest sporadic group for which a classification of all string C-group representations exists, is the third Conway group whose smallest permutation representation degree is on 276 points and whose Sylow 2-subgroups are of order $2^{10} = 1024$.

We use the technique described in~\cite{ConnorONan} and develop it further to obtain an algorithm that can be used on any permutation group as well as any matrix group to find polytopes of rank at least four. We also present a new technique to find polytopes of rank three.

With these new algorithms we managed to classify all string C-group representations of the sporadic groups $Suz$ and $Ru$. Although these groups have order smaller than the one of $Co_3$, their smallest permutation representation degrees are respectively 1782 and 4060, and their Sylow 2-subgroups are of respective orders $2^{13}$ and $2^{14}$.

The classifications are made available in the online version of the
Atlas~\cite{AtlasSmallGroups}.

The paper is organised as follows.
In Section~\ref{prelim}, we give the background needed to understand this paper.
In Section~\ref{algo}, we give some useful results to improve existing algorithms, as well as the pseudo-code of our new algorithm in rank at least four.
In Section~\ref{results}, we describe some of the results obtained with the new algorithm.

\section{String C-group representations}\label{prelim}

Let $G$ be a group and $S := \{\rho_0,\rho_1,...,\rho_{n-1}\}$ be a set of elements of $G$ such that $G = \langle S \rangle$.
We define the following two properties.
\begin{description}[style=unboxed,leftmargin=0cm]
\item[(C1)] the {\em string property}, that is $(\rho_i\rho_j)^{p_{ij}}=1_G$ where $1_G$ is the identity element of $G$, $p_{ii}=1$ for all $i \in \{ 0,1,...,n-1 \}$, $p_{ij}=2$ when $\mid i-j \mid \geq 2$ and $2 \leq p_{ji}=p_{ij} \leq \infty$ when $j=i-1$;
\item[(C2)] the {\em intersection property}, that is $\langle \rho_i \mid i \in I \rangle \cap \langle \rho_j \mid j \in J \rangle = \langle \rho_k \mid k \in I \cap J \rangle$ for any $I, J \subseteq \{ 0, 1,...,n-1 \}$
\end{description}
A pair $(G,S)$ as above that satisfies property $(C1)$ is called a {\em string group generated by involutions} (or sggi in short). A {\em string C-group representation} is an sggi that satisfies property $(C2)$. The {\em rank} of $(G,S)$ is the size of $S$.
For any subset $I\subseteq \{0, \ldots, n-1\})$, we denote by $G_I := \langle \rho_j : j \neq I\rangle$. If $I = \{i\}$, we define $G_{\{i\}}$ by $G_i$. Similarly, if $I = \{i,j\}$,  we denote $G_{\{i,j\}}$ by $G_{ij}$

As shown in~\cite{ARP}, 
the automorphism group $\Gamma(\mathcal{P})$ of an abstract regular polytope $\mathcal{P}$, together with the involutions that map a base flag to its adjacent flags is a string C-group representation (see~\cite[Propositions 2B8, 2B10 and 2B11]{ARP}). More precisely, if one fixes a base flag $\Phi$ of $\mathcal{P}$, $\Gamma(\mathcal{P})=\langle \rho_0,\rho_1,...,\rho_{n-1} \rangle$ with $\rho_i$ the unique involution such that $\rho_i(\Phi)=\Phi^i$ ($\Phi^i$ is the flag obtained from $\Phi$ by changing its element of rank $i$) then the pair $(\Gamma({\mathcal{P}}),\{\rho_0,\rho_1,...,\rho_{n-1}\})$ is a string C-group representation. Moreover, $\{ p_0,p_1,...,p_{n-1} \}$ where $\ord(\rho_{i-1}\rho_{i})=p_i$ being the Schl\"afli type of $\mathcal{P}$, we define the {\em Schl\"afli type} of a string C-group representation $(G,\{\rho_0, \ldots, \rho_{n-1}\})$ as the ordered set $\{ p_0,p_1,...,p_{n-1} \}$ where $p_i := \ord(\rho_{i-1}\rho_{i})$ for $i=1, \ldots, n-1$.

Conversely, as shown in~\cite[Theorem 2E11]{ARP}, a regular $n$-polytope can be constructed uniquely from a string C-group representation $(G,S)$ with  $S:=\{ \rho_0,...,\rho_{n-1}\}$. Let $G_i := \langle \rho_j \mid j\neq i \rangle$ for any $i \in \{ 0,1,...,n-1 \}$. We also set $G_{-1}=G_{n}:=G$.
For $i \in \{ -1,0,1,...,n-1,n \}$, we take the set of $i$-faces of $\mathcal{P}$ to be the set of all right cosets $G_i\varphi$ of $G_i$ in $G$. We define a partial order on $\mathcal{P}$ as follows : $G_i\varphi \leq G_j\psi$ if and only if $-1\leq i \leq j \leq n$ and $G_i\varphi \cap G_j\psi \neq \emptyset$.

We say that two string C-group representations $(G,S)$ and $(G,S')$ of $G$ are {\em isomorphic} if there exists an automorphism of $G$ that maps $S$ onto $S'$.
The {\em dual} of a string C-group representation $(G,\{\rho_0,\rho_1,...,\rho_{n-1}\})$ is the string C-group representation $(G,\{\rho_{n-1},\rho_{n-2},...,\rho_{0}\})$.

Due to the one-to-one correspondence between string C-group representations and automorphism groups of abstract regular polytopes, finding all abstract regular polytopes with a fixed automorphism group amounts to considering all ways of presenting a particular group as a string C-group and verifying whether they yield non-isomorphic abstract regular polytopes. 
Our algorithms classify string C-group representations of a group $G$ up to isomorphism and duality.

\section{Presentation of the new algorithms}\label{algo}
In this section, we present our improvements of the existing algorithms for computing all string C-group representations $(G,\{\rho_0, \ldots, \rho_{n-1}\})$ of a given group $G$ (up to isomorphism and duality).
\subsection{The large ranks}
We describe our technique to find string C-group representations of rank at least four. The algorithm uses the following observations.

\begin{lemma}\label{subcent}
Let $G$ be a group.
Let $(G,\{\rho_0, \ldots, \rho_{n-1}\})$ be a string C-group representation of $G$.
The subgroup $G_1$ is a subgroup of the centraliser of $\rho_0$, in particular $G_1 = \langle\rho_0\rangle \times G_{01}$ where $G_{01} \leq C_G(\rho_0)$.
\end{lemma}
\begin{proof}
This is a direct consequence of Proposition $2B12$ in \cite{ARP} and of the definition of string C-groups.
\end{proof}

This lemma implies that we can take for $\rho_0$ a representative of a conjugacy class of involutions of $G$, compute its centraliser $C_G(\rho_0)$ and then the subgroup $G_1$ must be a subgroup of $C_G(\rho_0)$.

We then proceed to construct $G_1$.
In order to do so, we produce all string C-group representations of subgroups $H$ of $C_G(\rho_0)$  that have $\rho_0$ as first generator.
As $C_G(\rho_0)$ is substantially smaller than $G$, one may now apply the former algorithms (from \cite{LastAlgo}) to $C_G(\rho_0)$ in order to classify all the string C-group representations of all its subgroups. With each string C-group representation $\langle \rho_0, \rho_2,...,\rho_{n-1} \rangle$ of $G_1$, we proceed by adding an appropriate involution $\rho_1$ of $G$ to it, giving a string C-group representation for $G$. 
In order to restrict the number of involutions $\rho_1$ to consider, we use the following observation.
\begin{lemma}\label{choose1}
If  $(G,\{\rho_0, \ldots, \rho_{n-1}\})$ is a string C-group representation of the group $G$, then $\rho_1\in C_G(\rho_3) \cap \ldots \cap C_G(\rho_{n-1})$.
\end{lemma}
\begin{proof}
Since $\rho_1$ has to commute with every generator $\rho_i$ for $i\geq 3$ by property (C2), we have that $\rho_1\in C_G(\rho_3) \cap \ldots \cap C_G(\rho_{n-1})$.
\end{proof}

We then check that the resulting pair $(G,\{ \rho_0, \rho_1,..., \rho_{n-1}\})$ is a string C-group representation of $G$ using the following proposition.

\begin{proposition}\label{checkpoly}\cite[Proposition 2E16(a)]{ARP}
Let $(G,\{ \rho_0, \rho_1,..., \rho_{n-1}\})$ be an sggi.
If $G_0$ and $G_{n-1}$ are string C-groups and if $G_{n-1} \; \cap\; G_0= G_{0,n-1}$ then $(G,\{ \rho_0, \rho_1,..., \rho_{n-1}\})$ is  a string C-group.
\end{proposition}

Finally, let us note that although the former algorithms only allowed to be implemented for groups of permutations, this new algorithm also works for matrix groups. This gives access to groups much larger in size, as long as the centralizers of involutions of $G$ are small enough to be treated as permutation groups.

 We give in Table~\ref{pseudocode} the pseudo-code of our new algorithm.

\begin{table}
\begin{lstlisting}
Input: $G$ the group for which we want to compute all 
       string C-group representations.
Output: $L$ a sequence containing the pairwise non-isomorphic
        string C-group representations.

Compute the conjugacy classes $C(G)$ of elements of $G$.
Initialise $L$.
For each conjugacy class $c$ of elements of order $2$ in $C(G)$:
  Let $r_0$ be a representative of that class.
  Build the centralizer $H$ of $r_0$ in $G$.
  If $G$ is a matrix group,
    find a permutation representation $P(H)$ of $H$;
    reduce the degree of $P(H)$;
  If $G$ is already a permutation group then
    only reduce its degree and call it $P(H)$.
  Using the existing procedures to do so (see [6]),
    compute the string C-group representations with generators
    $\{r_0, r_2, \ldots, r_{n-1}\}$ for subgroups of $C_G(r_0)$,
    forcing $r_0$ as first generator.
  For each such representation, try to complete it by inserting an
  involution $r_1$ between $r_0$ and $r_2$, using the fact that
  it has to be in the centralisers of the $r_i$'s for 
  $i=3, \ldots, n-1$.
  Compute $G_0=\langle r_1,...,r_{n-1} \rangle$, $G_{n-1}=\langle r_0,...,r_{n-2} \rangle$ and $G_{0,n-1}=\langle r_1,...,r_{n-2} \rangle$.
  Check the intersection property for $G_0$ and $G_{n-1}$ and
  Check that $G_0 \cap G_{n-1}=G_{0,n-1}$.
  Let $GP=\langle r_0,r_1,...,r_{n-1}\rangle$ and let $\overset{\sim}{GP}=\langle r_{n-1},...,r_1,r_0\rangle$ be its dual.
  If $GP$ and $\overset{\sim}{GP}$ are non-isomorphic to one of the elements of $L$,
    Add $GP$ to $L$.
At the end, $L$ contains one representative of each isomorphism
class of string C-group representation of $G$.
\end{lstlisting}
\caption{The pseudo-code of the new algorithm}\label{pseudocode}
\end{table}
\subsection{The rank three case}
An improvement can be made to compute the rank three string C-group representations of a given group $G$ as follows.

We first need to find one representative of each conjugacy class of dihedral subgroups. In order to do this, we compute the conjugacy classes of elements and for each representative $g$ of such a conjugacy class with $o(g)\geq 3$, we compute the normalizer $N := N_G(\langle g \rangle)$. Then we look for an element $h$ of $N$ such that $g^h = g^{-1}$. The elements $g$ and $h$ generate a dihedral group. Each conjugacy class of dihedral group can be found easily in that way. Let $S$ be a sequence containing one representative of each dihedral subgroup.

The next step is to store, for one representative $i_j$ of each conjugacy class $I_j$ of involutions of $G$ the set of involutions $C_j$ of $G$ that commute with $i$.

Then for each dihedral subgroup $D\in S$, and for each non-conjugate pair of generators $(\rho_0,\rho_1)$ of $D$, we look for involutions $\rho_2$ such that $\rho_2$ commutes with $\rho_0$ (those are in the set $C_j$ for the $j$ such that the involution $i_j$ is conjugate with $\rho_0$ in $G$) and such that $\langle \rho_0,\rho_1,\rho_2\rangle = G$. If $G$ is simple (or even if $G$ simply does not contain a nontrivial cyclic normal subgroup), the intersection property is satisfied and we do not need to test it (Proposition 4.1. in \cite{intersection}). Otherwise we do test it. We then do the same interverting the roles of $\rho_1$ and $\rho_0$.

\subsection{Computational times}

Table~\ref{tab:timings} gives, for a given group $G$, the Time\footnote{The timings presented in this section were obtained using {\sc Magma}\cite{magma} running on a computer with 32 cores running at 2.3Ghz and 384Gb of RAM. Note that {\sc Magma} does not do parallel computing, it is using one core at a time. } it takes (in seconds) to compute all string C-group representations with the new algorithm, the Time it takes with the old algorithm (when this time is small enough and the old algorithm was capable of getting the whole result), the number of string C-group representations of rank greater than 3 and the number of string C-group representations of rank three.
The time in the Time column is given as a sum of two numbers. The first number is the time it takes to compute string C-group representations of rank at least four. The second number is the time it takes to compute the rank three ones. So for instance, for the group $J_1$, it takes 1.4 seconds to get the two string C-group representations of rank four and show there are no string C-group representation of rank at least five, and it takes 2 seconds to get the 148 ones of rank three. 

\begin{table}
    \centering
\begin{tabular}{|c|c|c|c|c|}
\hline
$G$&Time&Time old algo&\# pol. of rank > 3&\# pol. of rank 3\\
\hline
\hline
$M_{11}$&0.01s + 0.02 = 0.03s&0.76&0&0\\
$M_{12}$&4.5s + 0.4 = 5s&191&14&23\\
$J_{1}$&1.4s + 2 = 3.4s&357&2&148\\
$M_{22}$&0.3s + 0.06 = 0.36s&414&0&0\\
\hline
$J_{2}$&14s + 3.2=17.2s&&17&137\\
$M_{23}$&0.45s + 0.07 = 0.52s&&0&0\\
$HS$&109 + 12 = 2m&&59&252\\
$J_{3}$&65 + 198 = 4.4m&&2 &303\\
$M_{24}$&1951 + 16 = 32.8m&&157&490\\
$McL$&3 + 1 = 4s&&0&0\\
$He$&25854 + 1166 = 7.5h&&76&1188\\
$Ru$&43516 + 114680 = 44h &&227&21594\\
$Suz$&5957 + 16014 = 6.1h&&270&7119\\
$ONan$&573962 + ???&&16&???\\
$Co_3$&28627 + 5522 = 9.5h &&895&10586\\
\hline

\end{tabular}
    \caption{Computing times for sporadic groups}
    \label{tab:timings}
\end{table}

\section{New results}~\label{results}

Two sporadic groups smaller than the $Co_3$ group had apparently never been investigated. The two main reasons were that these groups have a higher permutation representation degree than $Co_3$ and that they have larger Sylow 2-subgroups (of respective sizes $2^{13}$ and $2^{14}$ while the Sylow 2-subgroups of $Co_3$ have order $2^{10}$). We analysed them with our new set of programs and we summarize our findings here.

\subsection{The Rudvalis group}
The Rudvalis group was discovered by Arunas Rudvalis in 1973~\cite{rudvalis}. It has order $145,926,144,000 = 2^{14}\cdot 3^3\cdot 5^3\cdot 7\cdot 13\cdot 29$ and smallest permutation representation degree 4060. It has two conjugacy classes of involutions.
It has 21594 string C-group representations of rank three, 227 of rank four and none of higher rank.

\subsection{The Suzuki sporadic group}
The Suzuki sporadic group was discovered by Michio Suzuki in 1968~\cite{suzuki}. It has order $448,345,497,600 = 2^{13} \cdot 3^7 \cdot 5^2 \cdot 7 \cdot 11 \cdot 13$ and smallest permutation representation degree 1782. It has two conjugacy classes of involutions. It has 7119 string C-group representations of rank three, 257 of rank four, 13 of rank five and none of higher rank.


\begin{thebibliography}{10}

\bibitem{magma}
W.~Bosma, J.~Cannon, and C.~Playoust.
\newblock The {M}agma {A}lgebra {S}ystem {I}: the user language.
\newblock {\em J. Symbolic Comput.}, (3/4):235--265, 1997.

\bibitem{intersection}
M. Conder and D. Oliveros.
\newblock The intersection condition for regular polytopes.
\newblock {\em J. Combin. Theory Ser. A}, 120(6):1291--1304, 2013.

\bibitem{ConnorONan}
T.~Connor, D.~Leemans, and M.~Mixer.
\newblock Abstract regular polytopes for the O'Nan group.
\newblock {\em Internat. J. Algebra Comput.}, 24(01):59--68, 2014.




\bibitem{AtlasHartley}
M.~I. Hartley.
\newblock An atlas of small regular polytopes.
\newblock {\em Periodica Math. Hungar.}, 53(1--2):149--156, 2006.

\bibitem{HartleyHulpke}
M.~I. Hartley and A.~Hulpke.
\newblock Polytopes derived from sporadic simple groups.
\newblock {\em Contrib. Discrete Math.}, 5(2), 2010.

\bibitem{LeemansSurvey}
D.~Leemans.
\newblock String C-group representations of almost simple groups: a survey.
\newblock {\em Contemp. Math.}, to appear, arXiv 1910.08843.

\bibitem{LastAlgo}
D.~Leemans and M.~Mixer.
\newblock Algorithms for classifying regular polytopes with a fixed
  automorphism group.
\newblock {\em Contrib. Discrete Math.}, 7(2):105--118, 2012.


\bibitem{AtlasSmallGroups}
D.~Leemans and L.~Vauthier.
\newblock An atlas of abstract regular polytopes for small groups.
\newblock {\em Aequationes Math.}, 72(3):313--320, 2006.

\bibitem{ARP}
P.~McMullen and E.~Schulte.
\newblock {\em Abstract regular polytopes}, volume~92.
\newblock Cambridge University Press, 2002.

\bibitem{rudvalis}
A.~Rudvalis.
\newblock A new simple group of order $2^{14}\cdot3^3\cdot5^3 \cdot7 \cdot13
  \cdot29$.
\newblock {\em Notices Amer. Math. Soc.}, (2):A--95, 1973.

\bibitem{suzuki}
M.~Suzuki.
\newblock A simple group of order $448,345,497,600$.
\newblock In {\em Theory of Finite Groups (Symposium, Harvard Univ., Cambridge,
  Mass., 1968)}, Benjamin, New York, pages 113--119, 1969.

\end{thebibliography}
\end{document}